\documentclass[11pt,french,english]{amsart}
\usepackage[T1]{fontenc}
\usepackage[latin1]{inputenc}
\usepackage{latexsym}
\usepackage{amsthm}
\usepackage{amssymb}

\numberwithin{equation}{section} %% Comment out for sequentially-numbered
\numberwithin{figure}{section} %% Comment out for sequentially-numbered
\theoremstyle{plain}
\theoremstyle{plain}
\newtheorem{thm}{Theorem}
  \theoremstyle{plain}
  \newtheorem{prop}[thm]{Proposition}
  \theoremstyle{plain}
  \newtheorem{lem}[thm]{Lemma}
  \theoremstyle{remark}
  \newtheorem*{acknowledgement*}{Acknowledgement}

\AtBeginDocument{
  
}

\usepackage{babel}
\addto\extrasfrench{\providecommand{\fg}{\ifdim\lastskip>\z@\unskip\fi~\frqq}}

\begin{document}
\newcommand{\Alb}{{\rm Alb}}

\newcommand{\Jac}{{\rm Jac}}

\newcommand{\Hom}{{\rm Hom}}

\newcommand{\End}{{\rm End}}

\newcommand{\Aut}{{\rm Aut}}

\newcommand{\NS}{{\rm NS}}

\title[The Fano surface of the Fermat cubic threefold]{The Fano surface of the Fermat cubic threefold, the del Pezzo surface
of degree $5$ and a ball quotient.}

\author{Xavier Roulleau}
\begin{abstract}

We study in this paper a surface which has many intriguing and puzzling
aspects: on one hand it is related to the Fano surface of lines of
a cubic threefold, and the other hand it is related to a Ball quotient
occurring in the realm of hypergeometric functions, as studied by
Deligne and Mostow. It is moreover connected to a surface constructed
by Hirzebruch, in its works for constructing
surfaces with Chern Ratio equal to $3$ by arrangements of lines on
the plane. Furthermore, we obtain some results that are analogous
to the results of Yamasaki-Yoshida when they computed the lattice
of the Hirzebruch ball quotient surface. 
\end{abstract}

\maketitle


\section{Introduction}

Let us recall the following well known Theorem (see \cite{Kobayashi},
Theorem $2$) that relates an inequality between (log) Chern numbers
with the theory of Ball quotients:
\begin{thm}
\label{thm:Kobayashi}{[}Bogomolov, Hirzebruch, Miyaoka, Sakai, Yau].
Let $S$ be a smooth projective surface with ample canonical bundle
$K$ and let $D$ be a reduced simple normal crossing divisor on $S$
(may be $0$). Suppose that $K+D$ is nef and big. Then the following
inequality:\[
\overline{c}_{1}^{2}\leq3\overline{c}_{2}\]
holds, where $\bar{c}_{1}^{2},\bar{c}_{2}$ are the logarithmic Chern
numbers of $S'=S-D$ defined by: $\bar{c}_{1}^{2}=(K+D)^{2}$ and
$\bar{c}_{2}=\chi_{top}(S')$. Furthermore, the equality occurs if
and only if $S'$ is a ball quotient i.e., if and only if we obtain
$S'$ by dividing the ball  $\mathbb{B}_{2}$ with respect to a discrete
group $\Gamma$ of automorphisms acting on $\mathbb{B}_{2}$ properly
discontinuously and with only isolated fixed points. \\
If a smooth projective surface $X$ contains a rational curve and
its Chern numbers satisfies $c_{1}^{2}=3c_{2}>0$, then $X$ is the
projective plane.
\end{thm}
Few examples of surfaces with Chern ratio $\frac{c_{1}^{2}}{c_{2}}$
equals $3$ have been constructed algebraically i.e. by ramified covers
of known surfaces. The first examples are owed independently to Inoue
and Livné (for a reference, see \cite{Hulek}). Among these examples,
there is a surface $\mathbb{S}$, with Chern numbers: \[
c_{1}^{2}(\mathbb{S})=3c_{2}(\mathbb{S})=3^{2}5^{2},\]
that is the blow-down of $(-1)$-curves of a certain cyclic cover
of the Shioda modular surface of level $5$. Hirzebruch \cite{Hirzebruch}
then constructed other examples. Starting with the degree $5$ del
Pezzo surface $\mathcal{H}_{1}$ and for any $n>1$, he constructed
a cover $\eta_{n}:\mathcal{H}_{n}\rightarrow\mathcal{H}_{1}$ of degree
$n^{5}$ branched exactly over the ten $(-1)$-curves of $\mathcal{H}_{1}$,
with order $n$, such that for $n=5$:
\begin{thm}
\label{thm:(Hirzebruch).-The-Chern}{[}Hirzebruch]. The Chern numbers
of $\mathcal{H}_{5}$ satisfies: \[
c_{1}^{2}(\mathcal{H}_{5})=3c_{2}(\mathcal{H}_{5})=3^{2}5^{4}.\]

\end{thm}
Then Ishida \cite{Ishida} established a link between the surface
$\mathcal{H}_{5}$ and the Inoue-Livné surface $\mathbb{S}$:
\begin{prop}
\label{pro:There-is-a}{[}Ishida]. There is a étale map $\mathcal{H}_{5}\rightarrow\mathbb{S}$
that is a quotient of $\mathcal{H}_{5}$ by an automorphism group
of order $25$.
\end{prop}
In the present  paper, we give an example of a surface with log Chern
numbers satisfying $\bar{c}_{1}^{2}=3\bar{c}_{2}$, and we obtain
in part A) and B) of Theorem \ref{thm: principal} below, results
analogous to Theorem \ref{thm:(Hirzebruch).-The-Chern} and Proposition
\ref{pro:There-is-a}:\\
The variety that parametrizes the lines on a smooth complex cubic
threefold $F\hookrightarrow\mathbb{P}^{4}$ is a smooth surface of
general type called the Fano surface of $F$ \cite{Clemens}. Let
$S$ be the Fano surface of the Fermat cubic threefold:\[
F=\{x_{1}^{3}+x_{2}^{3}+x_{3}^{3}+x_{4}^{3}+x_{5}^{3}=0\}.\]
It is the only Fano surface that contains $30$ elliptic curves \cite{Roulleau} ;
the aim of this paper is to prove the following Theorem:
\begin{thm}
\label{thm: principal}A) There is an open subvariety $S'\subset S$,
complement of $12$ disjoint elliptic curves on $S$, such that $S'$
is a ball quotient with log Chern numbers $\bar{c}_{1}{}^{2}=3\bar{c}_{2}=3^{4}$.\\
B) There is a étale map $\kappa:\mathcal{H}_{3}\rightarrow S$
that is a quotient of $\mathcal{H}_{3}$ by an automorphism of order
$3$ and there is a degree $3^{4}$ ramified cover $\eta:S\rightarrow\mathcal{H}_{1}$
branched with order $3$ over the ten $(-1)$-curves of $\mathcal{H}_{1}$.\\
C) The surface $\mathcal{T}=\kappa^{-1}S'\subset\mathcal{H}_{3}$
is a ball quotient. Let $\Lambda$ be the lattice of $\mathcal{T}$,
i.e. the transformation group of the $2$-dimensional unit ball $\mathbb{B}_{2}$
such that $\Lambda\setminus\mathbb{B}_{2}$ is isomorphic to $\mathcal{T}$.
The lattice $\Lambda$ is the commutator group of the congruence group:
\[
\Gamma=\{T\in GL_{3}(\mathbb{Z}[\alpha])/\, T\equiv I\,\mbox{modulo}\,(1-\alpha)\,\mbox{and}{\,\,}^{t}\bar{T}HT=H\}\]
where $\alpha$ is a primitive third root of unity, $I$ is the identity
matrix and $H$ is Hermitian diagonal matrix with entries $(1,1,-1)$
defining the $2$-dimensional unit ball $\mathbb{B}_{2}$. 

D) The lattice $\Gamma$ is the Deligne-Mostow lattice associated
to the $5$-tuple $(1/3,1/3,1/3,1/3,2/3)$ (number $1$ in \cite{Deligne}
p. 86).
\end{thm}
We wish to remark that in order to prove the part B) and C), we use
a result of Namba on ramified Abelian covers of varieties that, to
the best of our knowledge, has never been used before this paper.

We wish also to remark that parts C) and D) of Theorem \ref{thm: principal}
is the very analog of the following result of Yamazaki and Yoshida
\cite{Yamazaki}:
\begin{thm}
{[}Yamazaki, Yoshida \cite{Yamazaki}, Theorem 1]. The lattice $\Lambda'$
of the ball quotient $\mathcal{H}_{5}$ is the commutator group of
the congruence group: \[
\Gamma'=\{T\in GL_{3}(\mathbb{Z}[\mu])/\, T\equiv I\,\mbox{modulo}\,(1-\mu)\,\mbox{and}{\,\,}^{t}\bar{T}HT=H\}\]
where $\mu$ is a primitive fifth root of unity, $I$ is the identity
matrix and $H$ is Hermitian diagonal matrix with entries $(1,1,(1-\sqrt{5})/2)$,
defining the $2$-dimensional unit ball $\mathbb{B}_{2}$. \\
The lattice $\Gamma'$ is the Deligne-Mostow lattice associated
to the $5$-tuple $(2/5,2/5,2/5,2/5,2/5)$ (number $4$ in \cite{Deligne}
p. 86).
\end{thm}
Let us explain how the Deligne-Mostow lattices occur. Let $\mu=(\mu_{1},\dots,\mu_{5})$
be a $5$-tuple of rational numbers with $0<\mu_{i}<1$ and $\sum\mu_{i}=2$.
Let $d$ be the $l.c.m.$ of the $\mu_{i}$ and let $n_{i}$ be such
that $\frac{n_{i}}{d}=\mu_{i}$. Let $M$ be the moduli space of $5$-tuples
$x=(x_{1},\dots,x_{5})$ of distinct points on the projective line
$\mathbb{P}^{1}$. For each point $x$ of $M$, and $1\leq i<j\leq5$,
we consider the periods:\[
\omega_{ij}=\int_{x_{i}}^{x_{j}}\frac{dz}{v}\]
on the curve $v^{d}=\Pi_{k=1}^{k=5}(z-x_{i})^{n_{k}}$. These (multivalued)
maps $\omega_{ij}$ clearly factor to $Q=M/Aut(\mathbb{P}^{1})$ (that
is isomorphic to the complement of the $10$ $(-1)$-curves of the
degree $5$ del Pezzo surface). They are called hypergeometric functions
and they satisfy what is called the Appell differential equations
system. It turns out that the $\omega_{ij}$ span a $3$ dimensional
vector space $W_{\mu}$, and these yield a multivalued holomorphic
map $Q\rightarrow\mathbb{P}^{2}=\mathbb{P}(W_{\mu}^{*})$. In fact,
the image of that map lies in the (copy of a) unit Ball $\mathbb{B}_{2}$
of $\mathbb{C}^{2}\subset\mathbb{P}^{2}$. The multivaluedness is
measured by the monodromy representation \[
\pi_{1}(Q)\rightarrow Aut(\mathbb{B}_{2})\]
whose image is denoted by $\Gamma_{\mu}$. The main results of the
fundamental papers of Deligne and Mostow \cite{Deligne} and Mostow \cite{Mostow}
are to prove that the group $\Gamma_{\mu}\subset PGL(W_{\mu}^{*})$
is discontinuous and acts as a lattice on $\mathbb{B}_{2}$ for only
a finite number of $5$-tuple $\mu$, to compute these $\mu$ and to provide examples
of non-arithmetic lattices acting on $\mathbb{B}_{2}$.

\section{The Fano surface of the Fermat cubic as a cover of $\mathcal{H}_{1}$.}

Let $S$ be the Fano surface of the Fermat cubic threefold: \[
F=\{x_{1}^{3}+x_{2}^{3}+x_{3}^{3}+x_{4}^{3}+x_{5}^{3}=0\}\hookrightarrow\mathbb{P}^{4}.\]
This surface is smooth, has Chern numbers $c_{1}^{2}=45,\, c_{2}=27$
and irregularity $5$ (see \cite{Clemens}, (0.7)). Let $A(3,3,5)\subset GL_{5}(\mathbb{C})$
be the group of diagonal matrices with determinant $1$ whose diagonal
elements are in $\mu_{3}:=\{x\in\mathbb{C}/x^{3}=1\}$. By Theorem
26 of \cite{Roulleau}, the automorphism group of $F$ is the semi-direct
product of the permutation group $\Sigma_{5}$ and $A(3,3,5)\simeq(\mathbb{Z}/3\mathbb{Z})^{4}$.
An automorphism $f$ of $F$ preserves the lines and induces an automorphism
on $S$ denoted by $\rho(f)$. Let $G$ be the group $\rho(A(3,3,5))$. 

Let $X$ be the quotient of $S$ by the group $G$ and let $\eta:S\rightarrow X$
be the quotient map.

\begin{prop}
\label{pro:Let X be} The surface $X$ is (isomorphic to) the del
Pezzo surface $\mathcal{H}_{1}$, and the cover $\eta$ is branched
with index $3$ over the ten $(-1)$-curves of $X$.
\end{prop}

\begin{proof}
Let us outline the proof of Proposition \ref{pro:Let X be}: using
classical results on the quotient of a surface by a group action,
we show that $X$ is a smooth surface, then we compute its Chern numbers,
and prove that the blowing down of four $(-1)$-curves on $X$ is
the plane, and that allows us to conclude that $X$ is the degree
$5$ del Pezzo surface $\mathcal{H}_{1}$.

In order to prove that the surface $X$ is smooth, we need to recall
two lemmas:\\
Let $s$ be a point of $S$. Let us denote by $T_{S,s}$ the tangent
space of $S$ at $s$, by $L_{s}\hookrightarrow F$ the line on $F\hookrightarrow\mathbb{P}^{4}=\mathbb{P}(\mathbb{C}^{5})$
corresponding to $s$ and by $P_{s}\subset\mathbb{C}^{5}$ the subjacent
plane to the line $L_{s}$. 
\begin{lem}
(\cite{Roulleau}, Proposition 12). \label{lem:Let-s-be}Let $s$
be a fixed point of an automorphism $\rho(f)$ ($f\in A(3,3,5)$).
The plane $P_{s}$ is stable under the action of $f$ and the eigenvalues
of \[
d\rho(f):T_{S,s}\rightarrow T_{S,s}\]
 are equal to the eigenvalues of the restriction of $f\in A(3,3,5)$
to the plane $P_{s}\subset\mathbb{C}^{5}$.
\end{lem}
Hence this Lemma gives us the action of the differential $d\rho(f)$
on the fixed points of $\rho(f)$. Recall:
\begin{lem}
(\cite{Roulleau}, Theorem 26). The Fano surface $S$ of the Fermat
cubic contains $30$ elliptic curves, denoted by $E_{ij}^{\beta}$
for indices $1\leq i<j\leq5,\,\beta\in\mu_{3}$. Each curve $E_{ij}^{\beta}$
parametrizes the lines on a cone in the cubic $F$. Their configuration
is as follows :

\[
E_{ij}^{\beta}E_{st}^{\gamma}=\left\{ \begin{array}{cc}
1 & \textrm{if }\{i,j\}\cap\{s,t\}=\emptyset\\
-3 & \textrm{if }E_{ij}^{\beta}=E_{st}^{\gamma}\\
0 & \textrm{otherwise}.\end{array}\right.\]

\end{lem}
The elements of the group $G$ have order $3$, for each of them it
is easy to compute it closed set of fixed points. Let $I$ be the
set of points $s$ in $S$ such that $s$ is an isolated fixed point
of an element of $G$ : $I$ is the set of the $135$ intersection
points of the $30$ elliptic curves. Let $i,j,s,t$ be indices such
that $\{i,j\}\cap\{s,t\}=\emptyset$ and let $s$ be the intersection
point of $E_{ij}^{1}$ and $E_{st}^{1}$. The orbit of $s$ by $G$
is the set of the $9$ intersection points of the curves $E_{ij}^{\beta}$
and $E_{st}^{\gamma}$, $\beta,\gamma\in\mu_{3}$. Let be $s\in I$,
the group: \[
G_{s}=\{g\in G/s\textrm{ is a isolated fixed point of }g\}\]
is isomorphic to $\mu_{3}^{2}$ and, by Proposition \ref{lem:Let-s-be},
its representation on the space $T_{S,s}$ is isomorphic to the representation:\[
(\alpha_{1},\alpha_{2})\in\mu_{3}^{2},\;\;\;(\alpha_{1},\alpha_{2}).(x,y)=(\alpha_{1}x,\alpha_{2}y)\in\mathbb{C}^{2}\]
on $\mathbb{C}^{2}$. That implies by \cite{Cartan}, that the image
of $s$ is a smooth point of $X$, thus $X$ is smooth. 

The ramification index of $\eta:S\rightarrow X$ at the points of
$I$ is $9$ and the ramification index of $\eta$ on the curve $E_{ij}^{\beta}$
is $3$. Let us denote by $K_{V}$ the canonical divisor of a surface
$V$. Let be $\Sigma=\sum_{i,j,\beta}E_{ij}^{\beta}$ ; the ramification
divisor of $\eta:S\rightarrow X$ is $2\Sigma$ and \[
K_{S}=\eta^{*}K_{X}+2\Sigma.\]
 By \cite{Clemens}, Lemma 8.1 and Proposition 10.21, we know moreover
that $\Sigma=2K_{S}$, hence $3^{4}(K_{X})^{2}=(\eta^{*}K_{X})^{2}=(-3K_{S})^{2}=9.45$
and $(K_{X})^{2}=5$.

The stabilizer in $G$ of an elliptic curve $E_{ij}^{\beta}\hookrightarrow S$
contains $27$ elements and the group that fixes each point of $E_{ij}^{\beta}$
has $3$ elements. Let $\eta_{ij}^{\beta}:E_{ij}^{\beta}\rightarrow X_{ij}$
be the restriction of $\eta$ to $E_{ij}^{\beta}$. The curve $X_{ij}$
is smooth because it is the quotient of a smooth curve by an automorphism
group. The map $\eta_{ij}^{\beta}$ is a degree $9$ ramified cover
over $3$ points with ramification index $3$, hence: \[
0=\chi_{top}(E_{ij}^{\beta})=9(\chi_{top}(X_{ij})-3)+3.3\]
and $\chi_{top}(X_{ij})=2$ : $X_{ij}$ is a smooth rational curve.
As: \[
\eta^{*}X_{ij}=3(E_{ij}^{1}+E_{ij}^{\alpha}+E_{ij}^{\alpha^{2}}),\]
we deduce that the $10$ curves $X_{ij}$ have the following configuration:
\[
X_{ij}X_{st}=\left\{ \begin{array}{cc}
1 & \textrm{if }\{i,j\}\cap\{s,t\}=\emptyset\\
-1 & \textrm{if }X_{ij}=X_{st}\\
0 & \mbox{otherwise.}\end{array}\right.\]
Let $I'=\eta(I)$ and let $\Sigma'=\sum X_{ij}$. By additive property
of the Euler characteristic, we have: \[
3^{3}=\chi_{top}(S)=3^{4}\chi_{top}(X-\Sigma')+3^{3}\chi_{top}(\Sigma'-I')+3^{2}\chi_{top}(I').\]
Moreover, $\chi_{top}(\Sigma')=5$ (for an example of such computation
see \cite{Sakai}) and we obtain $\chi_{top}(X)=7$. We can blow-down
four disjoint $(-1)$-curves among the ten curves $X_{ij}$ and we
obtain a surface with Chern numbers:\[
c_{1}^{2}=3c_{2}=9\]
but this surface contains $6$ rational curves. Hence, by Theorem
\ref{thm:Kobayashi}, it is the plane : $X$ is the blow-up of the
plane at four points. These points are in general position because
of the intersection numbers of the $X_{ij}$, therefore $X$ is the
degree $5$ del Pezzo surface $\mathcal{H}_{1}$ and the $X_{ij}$
are its ten $(-1)$-curves. 

We proved that the quotient map $\eta:S\rightarrow X=\mathcal{H}_{1}$
is an Abelian cover branched over the ten $(-1)$-curves of $X$ with
ramification index 3, and that completes the proof of Proposition
\ref{pro:Let X be}.
\end{proof}

Let us now prove that:
\begin{prop}
\label{pro:There-exists-a}There exists an étale map ${\kappa:\mathcal{H}}_{3}\rightarrow S$
of degree $3$. 
\end{prop}

\begin{proof}
To prove Proposition \ref{pro:There-exists-a}, we begin to recall
Namba's results on Abelian covers of algebraic varieties.

Let $D_{1},\dots,D_{s}$ be irreducible hypersurfaces of a smooth
projective variety $M$ and let $e_{1},\dots,e_{s}$ be positive integers.
A covering $\pi:Y\rightarrow M$ is said to branch (resp. to branch
at most) at $D=e_{1}D_{1}+\dots+e_{s}D_{s}$ if the branch locus is
(resp. is contained in) $\cup D_{i}$ and the ramification index over
$D_{i}$ is $e_{i}$ (resp. divides $e_{i}$). An Abelian covering
$\pi:Y\rightarrow M$ which branches at $D$ is said to be maximal
if for every Abelian covering $\pi_{1}:Y_{1}\rightarrow M$ which
branches at most at $D$, there is a map $\kappa:Y\rightarrow Y_{1}$
such that $\pi=\pi_{1}\circ\kappa$. Let :\[
Div^{0}(M,D)=\{\hat{E}=\frac{a_{1}}{e_{1}}D_{1}+\dots+\frac{a_{s}}{e_{s}}D_{s}+E/a_{i}\in\mathbb{Z},\, E\,\mbox{integral},\, c_{1}(\hat{E})=0\}.\]
We say that $F_{1},F_{2}\in Div^{0}(M,D)$ are linearly equivalent
$F_{1}\sim F_{2}$ if $F_{1}-F_{2}$ is integral and is a principal
divisor. The following result is due to Namba:
\begin{thm}
\label{thm de namba}(Namba \cite{Namba2}, Thm. 2.3.18.). There is
a bijective map of the set of (isomorphism classes of) Abelian coverings
$\pi:Y\rightarrow M$ branched at most at $D$ onto the set of finite
subgroups $\mathcal{G}$ of $Div^{0}(M,D)/\sim$. Let $G_{\pi}$ be
the transformation group of the cover $\pi:Y\rightarrow M$. The bijective
map satisfies:\\
(1) $G_{\pi}\simeq\mathcal{G}(\pi)$\\
(2) let $\pi_{1}:Y_{1}\rightarrow M$ and $\pi_{2}:Y_{2}\rightarrow M$
be Abelian covers branched at most at $D$. There is a map $\kappa:Y_{1}\rightarrow Y_{2}$
such that $\pi_{1}=\pi_{2}\circ\kappa$ if and only if $\mathcal{G}(\pi_{1})\subset\mathcal{G}(\pi_{2})$.
\end{thm}
One applies this Theorem to the ten $(-1)$-curves $X_{ij}$ of $X=\mathcal{H}_{1}$
with $e_{i}=3$. The group $Div^{0}(\mathcal{H}_{1},D)/\sim$ is isomorphic
to $(\mathbb{Z}/3\mathbb{Z})^{5}$ (for brevity, we skip the proof,
but for an example of such a computation, see the proof of Lemma \ref{lem:The-group-Namba}).
As $\eta_{3}:\mathcal{H}_{3}\rightarrow\mathcal{H}_{1}$ is a degree
$3^{5}$ Abelian cover branched over the $(-1)$-curves with index
$3$, the group $\mathcal{G}(\eta_{3})$ is equal to $Div^{0}(\mathcal{H}_{1},D)/\sim$.
Thus by Theorem \ref{thm de namba}, there exists $\kappa:\mathcal{H}_{3}\rightarrow S$
such that $\eta_{3}=\eta\circ\kappa$. As the maps $\eta$ and $\eta_{3}$
are branched with order $3$ over the ten $(-1)$-curves of $\mathcal{H}_{1}$,
the map $\kappa$ is étale. That completes the proof of Proposition \ref{pro:There-exists-a}.
\end{proof}

\section{The Fano surface of the Fermat cubic as a ball quotient.}

By \cite{Clemens}, Theorem 7.8, (9.14) and (10.11), the Fano surface
$S$ of the Fermat cubic is smooth with invariants $c_{1}^{2}=45$
and $c_{2}=27$. Let $S'\subset S$ be the complement of the union
$D$ of $12$ disjoints elliptic curves on $S$ (there are $5$ such
sets of $12$ elliptic curves, we can take by example the $12$ curves
$E_{1i}^{\beta},\,2\leq i\leq5,\,\beta^{3}=1$). Let $\overline{c}_{1}^{2},\overline{c}_{2}$
be the logarithmic Chern numbers of $S'$.
\begin{prop}
\label{cor:The-Chern-ratio}We have : $3\overline{c}_{2}=\overline{c}_{1}^{2}=81$,
therefore $S'$ is a ball quotient.\end{prop}
\begin{proof}
The canonical divisor $K_{S}$ of $S$ is ample, $K_{S}^{2}=45$ and
$K_{S}E=-E^{2}=3$ for an elliptic curve $E\hookrightarrow S$ (see
\cite{Clemens} (0.7) and \cite{Roulleau} Proposition 10), therefore
$K_{S}+D$ is nef. As $\bar{c}_{1}^{2}=(K_{S}+D)^{2}=45+2.12.3-12.3=81$,
the divisor $K_{S}+D$ is also big. As $\bar{c}_{2}(S')=e(S-D)=e(S)=27$,
we obtain $3\overline{c}_{2}=\overline{c}_{1}^{2}$. Because $D$
has no singularities, Theorem \ref{thm:Kobayashi} implies that $S'$
is a ball quotient.
\end{proof}
Let $H$ be the Hermitian diagonal matrix with entries $(1,1,-1)$
defining the $2$-dimensional unit ball $\mathbb{B}_{2}$ into $\mathbb{P}^{2}$.
Let $\alpha$ be a third primitive root of unity and let $\Gamma$
be the congruence group: \[
\Gamma=\{T\in GL_{3}(\mathbb{Z}[\alpha])/\, T\equiv I\,\mbox{modulo}\,(1-\alpha)\,\mbox{and}{\,\,}^{t}\bar{T}HT=H\},\]
where $I$ is the identity matrix. As $\kappa$ is étale and $S'$
is a ball quotient, the surface $\mathcal{T}=\kappa^{-1}S'\subset\mathcal{H}_{3}$
is a ball quotient. Let $\Lambda$ be the transformation group of
the $2$-dimensional unit ball $\mathbb{B}_{2}$ such that $\Lambda\setminus\mathbb{B}_{2}\simeq\mathcal{T}$.
We have:

\begin{thm}
\label{thm:The-group- lambda}The group $\Lambda$ is the commutator
group of $\Gamma$.
 \end{thm}

\begin{proof}

In order to compute $\Lambda$, we combine ideas in \cite{Yamazaki},
where Yamazaki and Yoshida computed the lattice of the Ball quotient
surface $\mathcal{H}_{5}$, and we use Namba's Theorem \ref{thm de namba}.

Let $\ell_{1},\dots,\ell_{6}\in H^{0}(\mathbb{P}^{2},\mathcal{O}(1))$
be the linear forms defining the $6$ lines on the plane going through
$4$ points in general position. Let $\mathcal{H}_{3}'$ be the normal
algebraic surface determined by the field \[
\mathbb{C}(\mathbb{P}^{2})((\frac{\ell_{2}}{\ell_{1}})^{1/3},\dots,(\frac{\ell_{2}}{\ell_{1}})^{1/3}).\]
It is an Abelian cover $\pi:\mathcal{H}_{3}'\rightarrow\mathbb{P}^{2}$ of degree $3^{5}$ of the plane branched with
order $3$ over the $6$ lines $\{\ell_{i}=0\}$ and the surface $\mathcal{H}_{3}$
is the fibered product of $\pi:\mathcal{H}_{3}'\rightarrow\mathbb{P}^{2}$
and the blow-up map $\tau:\mathcal{H}_{1}\rightarrow\mathbb{P}^{2}$ (see \cite{Sommese}
(1.3)). The situation is as follows:\[
\begin{array}{ccccc}
 & \stackrel{\kappa}{\swarrow} & \mathcal{H}_{3} & \stackrel{\iota}{\rightarrow} & \mathcal{H}'_{3}\\
S &  & \downarrow\eta_{3} &  & \downarrow\pi\\
 & \stackrel{\searrow}{\eta} & \mathcal{H}_{1} & \stackrel{\tau}{\rightarrow} & \mathbb{P}^{2}.\end{array}\]
We apply Namba's Theorem \ref{thm de namba} to the $6$ lines of
the complete quadrilateral on the plane, with weights $e_{i}=3$.

\begin{lem}
\label{lem:The-group-Namba}The group $Div^{0}(\mathbb{P}^{2},D)/\sim$
is isomorphic to $(\mathbb{Z}/3\mathbb{Z})^{5}$. 

\end{lem}

\begin{proof}
Let $L_{i}$ be the line $\{\ell_{i}=0\}$ and let $L$ be a generic
line. The group: \[
Div^{0}(\mathbb{P}^{2},D)/\sim=\{aL+\sum\frac{a_{i}}{3}L_{i}/a,a_{1},\dots,a_{6}\in\mathbb{Z}\mbox{ and }3a+\sum a_{i}=0\}/\sim.\]
 it is a sub-group of: \[
Div(\mathbb{P}^{2},D)/\sim=\{aL+\sum\frac{a_{i}}{3}L_{i}/a,a_{1},\dots,a_{6}\in\mathbb{Z}\}/\sim,\]
where the rational divisor $E=aL+\sum\frac{a_{i}}{3}L_{i}$ in $Div(\mathbb{P}^{2},D)$
is equivalent to $0$ if and only if the $a_{i},\,1\leq i\leq6$ are
divisible by $3$ and $c_{1}(E)=a+\mbox{\ensuremath{\frac{1}{3}}}\sum a_{i}=0$
(here we use that linear and numerical equivalences are equal on the
plane). The map :\[
\phi:\,\begin{cases}
\begin{array}{ccc}
Div(\mathbb{P}^{2},D)/\sim & \rightarrow & \mathbb{Z}\times(\mathbb{Z}/3\mathbb{Z})^{6}\\
aL+\sum\frac{a_{i}}{3}L_{i} & \rightarrow & (3a+\sum a_{i},\bar{a}_{1},\dots,\bar{a}_{6})\end{array}\end{cases}\]
is well defined and is an isomorphism. The group $Div^{0}(\mathbb{P}^{2},D)/\sim$
is isomorphic to:\[
\{(a,\bar{a}_{1},\dots,\bar{a}_{6})\in\mathbb{Z}\times(\mathbb{Z}/3\mathbb{Z})^{6}/a=0\mbox{ and }\sum\bar{a}_{i}=0\}\]
and is therefore isomorphic to $(\mathbb{Z}/3\mathbb{Z})^{5}$.
 \end{proof}

By Lemma \ref{lem:The-group-Namba} and Theorem \ref{thm de namba},  the degree $3^5$ Abelian cover $\pi:\mathcal{H}'_{3}\rightarrow\mathbb{P}^{2}$ is the maximal Abelian cover.

Let $b:\mathbb{P}^{2}\rightarrow\mathbb{N}$ be the function such
that $b(p)=1$ outside the complete quadrilateral, $b(p)=3$ on the
complete quadrilateral minus the $4$ triple points $p_{1},\dots,p_{4}$,
and $b(p)=\infty$ on these $4$ points. The pair $(\mathbb{P}^{2},b)$
is an orbifold (for the theory of orbifold we refer to \cite{Yoshida},
Chap. 5). By \cite{Holzaphel}, Chap. 5, the universal cover of that
orbifold is $\mathbb{B}_{2}$ with the transformation group $\Gamma$.
Therefore, a cover $Z\rightarrow\mathbb{P}^{2}$ with branching index
$3$ over the complete quadrilateral corresponds to a normal sub-group
$K$ of $\Gamma$ and $\Gamma/K$ is isomorphic to the group of transformation
of the covering $Z\rightarrow\mathbb{P}^{2}$.

If moreover, the cover $Z\rightarrow\mathbb{P}^{2}$ is Abelian, the
group $K$ contains the commutator group $[\Gamma,\Gamma]$, thus
$\mathbb{B}_{2}/[\Gamma,\Gamma]$ is the maximal Abelian cover of
$(\mathbb{P}^{2},b)$. We have seen that the cover $\pi:{\mathcal{H}'}_{3}\rightarrow\mathbb{P}^{2}$
of degree $3^{5}$ is maximal among Abelian covers of $(\mathbb{P}^{2},b)$,
thus the lattice of the ball quotient $\mathcal{T}$ is the commutator
group $[\Gamma,\Gamma]$. 
\end{proof}
Moreover, we remark that :
\begin{thm}
The lattice $\Gamma$ is the Deligne Mostow lattice number $1$ in
\cite{Deligne} p. 86.\end{thm}
\begin{proof}
This is the fact that the universal cover of the orbifold $(\mathbb{P}^{2},b)$
of the proof of Theorem \ref{thm:The-group- lambda} is $\mathbb{B}_{2}$
with the transformation group $\Gamma$. 

\end{proof}

\begin{acknowledgement*}
I wish to thank Amir Dzambic for stimulating discussions on this paper,
and Martin Deraux for its explanations of the numerous questions that
I had on the Deligne-Mostow ball lattices. Finally, I wish to thanks
the Max Planck Institute of Bonn and the Japan Society for Promotion
of Sciences for their financial support.
\end{acknowledgement*}

Xavier Roulleau

Graduate School of Mathematical Sciences, University of Tokyo, 3-8-1
Komaba, Meguro, Tokyo 153-8914, Japan 

roulleau@ms.u-tokyo.ac.jp
\end{document}